\documentclass[12pt,letterpaper]{article}

\usepackage{amssymb,amsfonts,amscd,amsthm}
\usepackage[all,arc]{xy}
\usepackage{enumerate, bbm}
\usepackage{mathrsfs}
\usepackage{tikz}
\usepackage{mathtools}
\usepackage{hyperref}
\usepackage{color}
\usepackage{graphicx}
\usepackage{enumitem} 
\graphicspath{ {TexImages/} }

\newtheorem{thm}{Theorem}
\newtheorem{cor}[thm]{Corollary}
\newtheorem{prop}[thm]{Proposition}
\newtheorem{lem}[thm]{Lemma}

\theoremstyle{definition}

\hypersetup{
	colorlinks,
	citecolor=blue,
	filecolor=blue,
	linkcolor=blue,
	urlcolor=blue,
	linktocpage
}

\setenumerate[1]{label=\thesection.\arabic*.} 
\setenumerate[2]{label*=\arabic*.} 
\setlist[enumerate]{itemsep=2ex, topsep=2ex} 
\setlist[itemize]{itemsep=2ex, topsep=2ex}

\newcommand{\sig}{\sigma}

\newcommand{\sub}{\subseteq}

\renewcommand{\c}[1]{\mathcal{#1}}

\newcommand{\mr}[1]{\mathrm{#1}}
\newcommand{\Eul}{\genfrac{\langle}{\rangle}{0pt}{}}

\newcommand{\Asc}{\mr{Asc}}
\newcommand{\asc}{\mr{asc}}

\title{Counting Labeled Threshold Graphs with Eulerian Numbers}
\author{Sam Spiro\footnote{Dept.\ of Mathematics, UCSD {\tt sspiro@ucsd.edu}. This material is based upon work supported by the National Science Foundation Graduate Research Fellowship under Grant No. DGE-1650112.}}
\date{\today}
\begin{document}
	\maketitle
\begin{abstract}
	A threshold graph is any graph which can be constructed from the empty graph by repeatedly adding a new vertex that is either adjacent to every vertex or to no vertices.  The Eulerian number $\genfrac{\langle}{\rangle}{0pt}{}{n}{k}$ counts the number of permutations of size $n$ with exactly $k$ ascents.  Implicitly Beissinger and Peled proved that the number of labeled threshold graphs on $n\ge 2$ vertices is
	\[\sum_{k=1}^{n-1}(n-k)\genfrac{\langle}{\rangle}{0pt}{}{n-1}{k-1}2^k.\]
	Their proof used generating functions.  We give a direct combinatorial proof of this result.
\end{abstract}
\section{Introduction}
This paper deals with threshold graphs, which can be defined recursively as follows.  The empty graph is the unique threshold graph on 0 vertices.  An $n$-vertex graph $G$ is a threshold graph if and only if it can be obtained by taking a threshold graph $G'$ on $n-1$ vertices and adding a new vertex which is either isolated or adjacent to every other vertex of $G'$.  

Threshold graphs were first studied by Chv\'{a}tal, and Hammer \cite{CH} in relation to linear programming, and since then they have been extensively studied.  One such reason for this is that threshold graphs can be characterized in several different ways.  For example, $G$ is a threshold graph if and only if it contains no induced subgraph isomorphic to $2K_2$, $P_4$, or $C_4$ \cite{MP}.  Variations such as random threshold graphs \cite{D} and oriented threshold graphs \cite{B} have been studied in recent years.  We refer the reader to the book ``Threshold Graphs and Related Topics'' \cite{MP} for more information and characterizations of threshold graphs.

It is easy to prove that the number of unlabeled threshold graphs on $n$ vertices is exactly $2^{n-1}$.  Let $t_n$ denote the number of labeled threshold graphs on $n$ vertices.   Beissinger and Peled found the exponential generating function of $t_n$ to be $e^x(1-x)/(2-e^x)$ \cite{BP}.  Using this they were able to derive an asymptotic formula for $t_n$, and implicitly they found an exact formula for $t_n$ in terms of the Eulerian numbers $\Eul{n}{k}$, which we shall now define.

Let $\c{S}_n$ denote the set of permutations of size $n$, where we treat our permutations as words written in one line notation.  Given $\pi \in \c{S}_n$, we say that position $i$ with $1\le i\le n-1$ is an ascent of $\pi$ if $\pi_i<\pi_{i+1}$.  Let $\Asc(\pi)$ denote the set of ascents of a permutation $\pi$ and let $\asc(\pi)=|\Asc(\pi)|$.  Define the Eulerian number $\Eul{n}{k}$ to be the number of permutations $\pi\in \c{S}_n$ with $\asc(\pi)=k$.  With this, a formula for $t_n$ can be stated as follows.

\begin{thm}[\cite{BP}]\label{T-Main}
	For $n\ge 2$, the number of labeled threshold graphs on $n$ vertices is
	\[\sum_{k=1}^{n-1}(n-k)\Eul{n-1}{k-1}2^k.\]
\end{thm}

This result can be derived from (16) of Beissinger and Peled~\cite{BP} through some algebraic manipulation, though it is not immediately obvious that this is the case.  Here we give a more direct and combinatorial proof of Theorem~\ref{T-Main}.

\section{Proof of Theorem~\ref{T-Main}}
We will say that a pair $(\pi,w)$ is a threshold pair (of size $n$) if $\pi\in \c{S}_n$ and if $w$ is a word in $\{+1,-1\}^{n}$.  Given a threshold pair $(\pi,w)$, let $T(\pi,w)$ denote the labeled threshold graph obtained as follows.  Let $G_1$ be the graph with a single vertex $\pi_1$.  Given $G_{i-1}$ with $2\le i\le n$, define $G_i$ by introducing a new vertex to $G_{i-1}$ labeled $\pi_i$ that is either connected to every vertex of $G_{i-1}$ if $w_i=+1$, and otherwise $\pi_i$ is an isolated vertex.  We then let $T(\pi,w)=G_n$.  As an example, Figure~\ref{F-G} shows $\tilde{G}:=T(24135,++---)$, where for ease of notation we have omitted the 1's in $w$.  We will use $\tilde{G}$ as a working example throughout this paper.

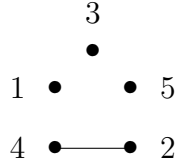
\begin{figure}
	\centering
	\begin{tikzpicture}[scale=1]
	\node at (0,.2) {$\bullet$};
	\node at (1,.2) {$\bullet$};
	\node at (0,1) {$\bullet$};
	\node at (1,1) {$\bullet$};
	\node at (.5,1.5) {$\bullet$};
	\node at (-.5,.2) {4};
	\node at (1.5,.2) {2};
	\node at (1.5,1) {5};
	\node at (-.5,1) {1};
	\node at (.5,2) {3};
	
	\draw (0,.2) -- (1,.2);
	\end{tikzpicture}
	\caption{$\tilde{G}=T(24135,++---)$} \label{F-G}
\end{figure}

There are several ways to write $\tilde{G}$, for example, $\tilde{G}=T(42351,-+---)$.  We wish to standardize our choice of a threshold pair.  To this end, we will say that a threshold pair $(\pi,w)$ of size $n\ge2$ is in standard form if $w_1=w_2$ and if $w_i=w_{i+1}$ implies $\pi_i<\pi_{i+1}$ for all $1\le i<n$.  For example, $(42351,-+---)$ is not in standard form but $(24135,++---)$ is. Our first goal will be to prove the following.

\begin{lem}\label{L-Stand}
	Let $G$ be a labeled threshold graph on $n\ge2$ vertices.  Then there exists a unique threshold pair $(\pi,w)$ in standard form such that $G=T(\pi,w)$.
\end{lem}

To prove this, we require two more lemmas.

\begin{lem}\label{L-Tech}
	Let $(\pi,w)$ and $(\sig,u)$ be threshold pairs of size $n\ge 2$ and let $G_1:=T(\pi,w)$ and $G_2:=T(\sig,u)$.  Then $G_1=G_2$ as labeled graphs if and only if the following two conditions hold.
	\begin{itemize}
		\item[(a)] $w_k=u_k$ for all $k\ge 2$.
		\item[(b)] For every $1\le i\le n$, if $j=\pi_i^{-1}$ and $k=\sig_i^{-1}$, then either $1\in \{j,k\}$ and $w_\ell=w_{\max\{j,k\}}$ for all $1< \ell\le \max\{j,k\}$, or for every $\ell$ with $\min\{j,k\}\le \ell \le \max\{j,k\}$ we have $w_\ell=w_j=w_k$.
	\end{itemize}
\end{lem}
\begin{proof}
	We first show that these conditions are necessary. We claim that condition (a) is necessary to have $G_1$ isomorphic to $G_2$, which certainly implies that (a) is necessary for $G_1$ and $G_2$ to be equal as labeled graphs.  This claim is true when $n=2$.  Assume the claim has been proven up to some $n\ge 3$.  If $w_n\ne u_n$, then exactly one of $G_1$ and $G_2$ will have an isolated vertex, so they cannot be isomorphic.  Otherwise let $G'_1$ be $G_1$ after deleting vertex $\pi_n$ and $G_2'$ be $G_2$ after deleting $\sig_n$.  Note that in both cases we either delete an isolated vertex or a vertex adjacent to every other vertex since $w_n=u_n$.  Thus $G_1\cong G_2$ if and only if $G_1'\cong G_2'$.  The result follows by applying the inductive hypothesis to $G_1'$ and $G_2'$ since the words generating these graphs are the words $w$ and $u$ after deleting their last letters.  Thus (a) is necessary.
	
	We  next show that (b) is necessary.  Assume for contradiction that $G_1=G_2$ and that (b) does not hold for some $i$.  By the above claim, we can assume that (a) holds.  Let $j=\pi_i^{-1}$ and $k=\sig_i^{-1}$.  If $j=k$ then (b) holds, a contradiction.  Thus we can assume that $j\ne k$, and without loss of generality we can assume $j<k$.  Let $d_r$ be the degree of vertex $i$ in $G_r$ for $r=1,2$.  First consider the case $j=1$ and $w_k=+1$.  In this case $d_1=|\{\ell:w_\ell=+1,\ \ell>1\}|$ and $d_2=k-1+|\{\ell:w_\ell=+1,\ \ell>k\}$, where we used that $u_\ell=w_\ell$ for all $\ell>1$ by (a).  Thus $d_1<d_2$ unless $w_\ell=+1$ for all $1<\ell\le k$. Because $G_1=G_2$, this must be the case, so (b) holds for $i$, a contradiction.    Essentially the same proof works if $j=1$ and $w_k=-1$.
	
	Now assume $j>1$ and $w_j=+1$, so $d_1=j-1+|\{\ell:w_\ell=+1,\ \ell>j\}|$.  If $w_k=-1$ then $d_2=|\{\ell:w_\ell=+1,\ \ell>k\}|<d_1$ since $j-1\ge 1$ by assumption.  In this case we cannot have $G_1=G_2$, so we can assume $w_k=+1$.  This implies $d_2=k-1+|\{\ell:w_\ell=+1,\ \ell>k\}|$. This will be strictly larger than $d_1$ unless $w_\ell=+1$ for all $j<\ell<k$.  Thus (b) holds for $i$, a contradiction.  Essentially the same proof works if $j>1$ and $w_j=+1$.  We conclude that (b) is necessary.
	
	To show that these conditions are sufficient, let $(\pi,w)$ and $(\sig,u)$ be threshold pairs satisfying (a) and (b).  Fix some $i$ and let $j=\pi_i^{-1}$ and $k=\sig_i^{-1}$.  We can assume without loss of generality that $j\le k$.  First consider the case $j=1$ and $w_k=+1$.  Then the neighborhood of $i$ in $G_1$ is $\{\pi_{\ell}^{-1}: \ell>j,\ w_{\ell}=+1\}$, and the neighborhood of $i$ in $G_2$ is $\{\pi_{\ell}^{-1}:\ell<k\}\cup \{\pi_{\ell}^{-1}: \ell>k,\ w_{\ell}=+1\}$, where again we used that $u_{\ell}=w_{\ell}$ for all $\ell>1$.  By (b), $w_{\ell}=+1$ for all $1<\ell<k$, so these two sets are equal.  The same result holds if $j=1$ and $w_k=-1$.
	
	Assume $j>1$ and $w_j=+1$.  We have $u_k=w_k=w_j=+1$ by (a) and (b).  Thus the neighborhood of $i$ in $G_1$ is $\{\pi_{\ell}^{-1}:\ell<j\}\cup \{\pi_{\ell}^{-1}: \ell>j,\ w_{\ell}=+1\}$, and the neighborhood of $i$ in $G_2$ is $\{\pi_{\ell}^{-1}:\ell<k\}\cup \{\pi_{\ell}^{-1}: \ell>k,\ w_{\ell}=+1\}$.  By (b) we have $w_\ell=+1$ for all $j<\ell<k$, so these sets are equal. The same result holds if $j>1$ and $w_j=-1$.  We conclude that the neighborhoods of every vertex is the same in both $G_1$ and $G_2$, and hence $G_1=G_2$.
\end{proof}

\begin{lem}\label{L-Ex}
	If $G$ is a threshold graph on $n\ge 2$ vertices, then there exists a threshold pair $(\pi,w)$ such that $G=T(\pi,w)$.
\end{lem}
\begin{proof}
	This certainly holds when $n=2$, so assume it holds up to some $n\ge 3$.  Because $G$ is a threshold graph, there exists a labeled threshold graph $H$ on $n-1$ vertices such that $G$ is isomorphic to $H$ together with the additional vertex $n$ which is either isolated or adjacent to every other vertex of $G'$.  Denote this labeled graph that $G$ is isomorphic to by $K$.  
	
	By our inductive hypothesis, $H=T(\pi',w')$ for some threshold pair $(\pi',w')$.  Define $\pi$ by $\pi_k=\pi'_k$ for $k<n$ and $\pi_n=n$.  Define $w$ by  $w_k=w'_k$ for $k<n$ with $w_n=-1$ if $K$ contains an isolated vertex and $w_n=+1$ otherwise.  Then $K=T(\pi,w)$.  By construction there exists a graph isomorphism $\sig:V(K)\to V(G)$.  Thus $T(\sig\circ \pi,w)$ is isomorphic to $G$ with the identity map serving as the graph isomorphism.  In other words, $G=T(\sig\circ \pi,w)$.
\end{proof}

\begin{proof}[Proof of Lemma~\ref{L-Stand}]
	We first show that such a pair exists. Let $(\pi',w')$ be a threshold pair with $G=T(\pi',w')$, which exists by Lemma~\ref{L-Ex}.  Define $w$ by $w_k=w_k'$ for $k>1$ and $w_1=w_2$.  Note that $T(\pi',w)=G$ by Lemma~\ref{L-Tech}.  Next define $\pi$ by repeatedly flipping adjacent letters of $\pi'$ that are out of order.  More precisely, let $\pi^{(0)}=\pi'$.  Inductively assume we have defined $\pi^{(j)}$.  If $(\pi^{(j)},w)$ is in standard form, take $\pi=\pi^{(j)}$.  Otherwise there exists some index $i$ such that $\pi^{(j)}_i>\pi^{(j)}_{i+1}$ and $w_i=w_{i+1}$.  Define $\pi^{(j+1)}$ by $\pi^{(j+1)}_{i}=\pi^{(j)}_{i+1},\ \pi^{(j+1)}_{i+1}=\pi^{(j)}_i$, and with $\pi^{(j+1)}_k=\pi^{(j)}_k$ for all other $k$.  Note that this process eventually terminates (this can be seen, for example, by noting that the number of inversions decreases at each step), and that $T(\pi^{(j+1)},w)=T(\pi^{(j)},w)$ for all $j$ by Lemma~\ref{L-Tech}.  As $T(\pi^{(0)},w)=T(\pi',w)=G$, we conclude that $T(\pi,w)=G$, and hence such a pair exists.
	
	To show that this pair is unique, assume that $(\sig,u)$ is also a threshold pair in standard form with $G=T(\sig,u)$.  By Lemma~\ref{L-Tech} we must have $u_k=w_k$ for all $k>1$.  Further, $u_1=u_2=w_2=w_1$ since the pairs are in standard form.  We next partition $w$ into maximal segments that are all $\pm 1$.  To this end, let $p_0=1$.  Inductively given $p_{r-1}$, define $p_{r}$ to be the smallest integer $p$ such that $w_{p}\ne w_{p_{r-1}}$, and let $p_r=n+1$ if no such integer exists.  Define $P_r=\{\pi_i:p_r\le i<p_{r+1}\}$ and $S_r=\{\sig_i:p_r\le i<p_{r+1}\}$.  
	
	We claim that $P_r=S_r$ for all $r$.  Indeed, assume that there exists some $i\in P_r$ and $i\in S_{r'}$ with, say, $r<r'$.  Let $j=\pi_i^{-1}$ and $k=\sig_i^{-1}$.  By Lemma~\ref{L-Tech} we have $w_\ell=w_j$ for all $j\le \ell \le k$.  In particular this holds for $\ell=p_{r+1}$ since $j<p_{r+1}\le k$, which is a contradiction since $w_{p_r}=w_j$ by assumption of $\pi_j\in P_r$.  We conclude that $P_r=S_r$ for all $r$.  Because $(\pi,w)$ is in standard form, we also must have $\pi_{p_r}<\pi_{p_r+1}<\cdots<\pi_{p_{r+1}-1}$ for all $r$, and the same inequalities hold with $\pi$ replaced by $\sig$.  We conclude that $\pi_i=\sig_i$ for all $p_r\le i<p_{r+1}$ for all $r$, and hence $\pi=\sig$, proving the result.
\end{proof}

We now define our sets for the desired bijection.  Let $T_n$ denote the set of labeled threshold graphs on $n$ vertices.  Let $\c{S}_n^+$ for $n\ge 2$ be the set of permutations of length $n$ with $\pi_1<\pi_2$.  That is, these are the set of permutations which begin with an ascent.  Define $P_n:=\{(\pi,A):\pi \in \c{S}_n^+,\ A\sub \Asc(\pi)\}$.

\begin{prop}\label{P-Bi}
	There exists a bijection from $T_n$ to $P_n$.
\end{prop}
\begin{proof}	
	Let $G$ be a labeled threshold graph and $(\pi^G,w^G)$ the unique threshold pair guaranteed by Lemma~\ref{L-Stand}.  Define $A'_G=\{i:w_i^{G}=w_{i+1}^G,\ 2\le i\le n-1\}$.  Let $A_G=A_G'\cup \{1\}$ if $w_1=+1$ and let $A_G=A_G'$ if $w_1=-1$.  Define $\phi(G)=(\pi^G,A_G)$.  For example, if $\tilde{G}$ is as in Figure~\ref{F-G}, we have $(\pi^{\tilde{G}},w^{\tilde{G}})=(24135,++---)$, and hence $\phi(\tilde{G})=(24135,\{1,3,4\})$.  We claim that the map $\phi$ gives the desired bijection.
	
	We first show that $\phi$ is a map from $T_n$ to $P_n$.  Indeed, because $(\pi^G,w^G)$ is in standard form, we have $w_1^G=w_2^G$ and hence $\pi_1^G<\pi_2^G$, so $\pi^G\in \c{S}_n^+$.  By similar reasoning we find that $A_G\sub \Asc(\pi^G)$, proving the claim.
	
	Let $(\pi,A)$ be an element of $P_n$.  We define the word $w$ as follows.  Let $w_1=w_2=+1$ if $1\in A$ and set $w_1=w_2=-1$ otherwise.  Given $w_{k}$, let $w_{k+1}=w_{k}$ if $k\in A$ and otherwise let $w_{k+1}=-w_{k}$.   We claim that $G=T(\pi,w)$ is the unique threshold graph with $\phi(G)=(\pi,A)$.  
	
	First observe that $A\sub \Asc(\pi)$ implies the pair $(\pi,w)$ is in standard form. Thus $\phi(G)=(\pi,A_G)$, and it is not difficult to verify that $A_G=A$ by construction, so $\phi(G)=(\pi,A)$.  Assume that $H$ is also such that $\phi(H)=(\pi,A)$, so in particular $\pi^H=\pi$.  We claim that $w_k^H=w_k$ for all $k$.  Indeed, because $A_H=A$, we must have $w_1^H=w_1$, as this completely determines whether $1$ is in $A_H$ or not, and also $w_2^H=w_1^H=w_1=w_2$ since both pairs are in standard form.  Inductively assume that $w_{k}^H=w_k$ for some $2\le k\le n-1$.  If $k\in A$, then we must have $w_{k+1}^H=w_k^H=w_k=w_{k+1}$, and otherwise we have $w_{k+1}^H=-w_{k}^H=-w_{k}=w_{k+1}$.  We conclude the result by induction.  Thus $(\pi^H,w^H)=(\pi,w)=(\pi^G,w^G)$, and we conclude that $H=G$ by Lemma~\ref{L-Stand}.  Thus each element of $P_n$ is mapped to by a unique element of $T_n$ and the result follows.
\end{proof}

All that remains is to enumerate $P_n$.  To this end, we say that a permutation $\pi$ has a descent in position $i$ if $\pi_i>\pi_{i+1}$.   

\begin{lem}\label{L-Count}
	For all $n$ and $d$ with $n\ge 1$ and $0\le d\le n-1$, let $\c{S}_{n,d}^+$ be the set of permutations of size $n$ which begin with an ascent and which have exactly $d$ descents.  If $P(n,d):=|\c{S}_{n,d}^+|$, then  \[P(n,d)=(d+1)\Eul{n-1}{d}.\]
\end{lem}
We note that this result is proven in \cite{S}, but for completeness we include the full proof here.  For this proof, we recall the following recurrence for the Eulerian numbers,  which is valid for all $n\ge 1$ and $d\ge 0$ after adopting the convention $\Eul{0}{d}=0$ for $d>0$, $\Eul{0}{0}=1$, and $\Eul{n}{-1}=0$ \cite{Con}:
\begin{align}\label{E-Eulerian}
\Eul{n}{d}=(d+1)\Eul{n-1}{d}+(n-d)\Eul{n-1}{d-1}.
\end{align}

\begin{proof}
	The result is true for $d=0$, so assume $d\ge 1$.  For any fixed $d$ the result is true for $n=1$, so assume $n\ge 2$. 
	To help us prove the result, we define $\c{S}_{n,d}^-$ to be the set of permutations which begin with a descent and which have exactly $d$ descents.  Define $M(n,d):=|\c{S}_{n,d}^-|$.  By construction we have
	\begin{equation}\label{E-MP}
	P(n,d)+M(n,d)=\Eul{n}{d}.
	\end{equation}
	
	Define the map $\phi:\c{S}_{n,d}^+\to \c{S}_{n-1}$ by sending $\pi \in \c{S}_{n,d}^+$ to the word obtained by removing the letter $n$ from $\pi$.  We wish to determine the image of $\phi$.  Let $\pi$ be a permutation in $\c{S}_{n,d}^+$, and let $i$ denote the position of $n$ in $\pi$.  Note that $i\ne 1$ since $\pi$ begins with an ascent.  If $i=n$ or $\pi_{i-1}>\pi_{i+1}$ with $i>2$, then $\phi(\pi)$ will continue to have $d$ descents and begin with an ascent, so $\phi(\pi)\in \c{S}_{n-1,d}^+$.  If $i=2$ and $\pi_1>\pi_3$, then  $\phi(\pi)\in \c{S}_{n-1,d}^-$.  If $\pi_{i-1}<\pi_{i+1}$, then $\phi(\pi)\in \c{S}_{n-1,d-1}^+$.  
	
	It remains to show how many times each element of the image is mapped to by $\phi$.  If $\pi\in \c{S}_{n-1,d}^+$, then $n$ can be inserted into $\pi$ in $d+1$ ways to obtain an element of $\c{S}_{n,d}^+$ (it can be placed at the end of $\pi$ or in between any $\pi_{i}>\pi_{i+1}$).  If $\pi \in \c{S}_{n-1,d-1}^+$, then $n$ can be inserted in $\pi$ in $n-d$ ways to obtain an element of $\c{S}_{n,d}^+$ (it can be placed in between any $\pi_i<\pi_{i+1}$).  If $\pi\in \c{S}_{n-1,d}^-$, then $n$ must be inserted in between $\pi_1>\pi_2$ in order to have the word begin with an ascent.  With this and the inductive hypothesis, we conclude that
	\begin{align}
	P(n,d)&=(d+1)P(n-1,d)+(n-d)P(n-1,d-1)+M(n-1,d)\nonumber\\ 
	&=(d+1)^2\Eul{n-2}{d}+(n-d)d\Eul{n-2}{d-1}+M(n-1,d).\label{E-AEul}
	\end{align}
	By using \eqref{E-MP}, the inductive hypothesis, and \eqref{E-Eulerian}; we find \begin{align*}M(n-1,d)&=\Eul{n-1}{d}-P(n-1,d)\\ &=\Eul{n-1}{d}-(d+1)\Eul{n-2}{d}\ \\ &=(n-d)\Eul{n-2}{d-1}.\end{align*}
	Substituting this into \eqref{E-AEul} and applying \eqref{E-Eulerian} again gives the result.
\end{proof}

\begin{cor}\label{C-Count}
	Let $n\ge 2$ and $1\le k\le n-1$.  The number of permutations of $\c{S}_n^+$ with exactly $k$ ascents is $(n-k)\Eul{n-1}{k-1}$.
\end{cor}
Note that there is no need to consider $k=0$ as every permutation of $\c{S}_n^+$ automatically has at least one ascent.

\begin{proof}
	This quantity is exactly $(n-k)\Eul{n-1}{n-1-k}$ by Lemma~\ref{L-Count} after replacing $d$ with $n-1-k$ (as any permutation of size $n$ with $k$ ascents has $n-1-k$ descents).  It is well known and easy to prove that $\Eul{m}{x}=\Eul{m}{m-1-x}$ for $m>0$ \cite{Con}, from which the result follows.
\end{proof}

\begin{proof}[Proof of Theorem~\ref{T-Main}]
	By Proposition~\ref{P-Bi} it is enough to prove that $P_n$ has this cardinality.  Given $\pi\in \c{S}_n^+$, the number of pairs $(\pi,A)\in P_n$ is exactly $2^{\asc(\pi)}$.  By Corollary~\ref{C-Count} we conclude that
	\[
		|P_n|=\sum_{k=1}^{n-1} (n-k)\Eul{n-1}{k-1}2^k,
	\]
	proving the result.
\end{proof}

\end{document}